\documentclass[12pt, a4paper]{article}
\usepackage{amsmath, amssymb, amsthm}
\usepackage{graphics}
\usepackage{graphicx}
\usepackage{epsfig}
\usepackage{dblfloatfix}
\usepackage{float}
\usepackage{placeins}
\usepackage{flafter}
\usepackage[usenames]{color}
\usepackage{cite}

\usepackage{epstopdf}
\usepackage{url}
\usepackage{hyperref}
\date{}
\newtheorem{df}{Definition}[section]
\newtheorem{pro}{Proposition}[section]
\newtheorem{thm}{Theorem}[section]
\newtheorem{cor}{Corollary}[section]
\newtheorem{ex}{Example}[section]
\newtheorem{lem}{Lemma}[section]
\newtheorem{rem}{Remark}[section]

\makeatletter

\renewcommand\section{\@startsection {section}{1}{\z@}
{-30pt \@plus -1ex \@minus -.2ex} {2.3ex \@plus.2ex}
{\normalfont\normalsize\bfseries}}

\renewcommand\subsection{\@startsection{subsection}{2}{\z@}
{-3.25ex\@plus -1ex \@minus -.2ex} {1.5ex \@plus .2ex}
{\normalfont\normalsize\bfseries}}

\renewcommand{\@seccntformat}[1]{\csname the#1\endcsname. }

\makeatother

\begin{document}
\title{Extension of Almost Armendariz Rings}
\author{\small{Sushma Singh and Om Prakash} \\ \small{Department of Mathematics} \\ \small{Indian Institute of Technology Patna, Bihta} \\ \small{Patna, India- 801 106} \\  \small{sushmasingh@iitp.ac.in \& om@iitp.ac.in }}

\date{}

\maketitle
\begin{abstract}
 A ring $R$ is said to be an almost Armendariz ring if whenever product of two polynomials in $R[x]$ is zero, then product of their coefficients are in $N_{*}(R)$. In this article, for an endomorphism $\alpha$ on $R$, we define an $\alpha$-almost Armendariz ring of $R$ considering the polynomials in skew polynomial ring $R[x; \alpha]$ instead of $R[x]$. It is the generalisation of an almost Armendariz ring\cite{S} and an $\alpha$-Armendariz ring \cite{CC}. Moreover, for an endomorphism $\alpha$ of $R$, we define an $\alpha$-skew  almost Armendariz ring, and prove that a reversible ring $R$ with certain condition on endomorphism $\alpha$, its polynomial ring $R[x]$ is an $\overline{\alpha}$-skew almost Armendariz ring.
\end{abstract}


\section{Introduction}
In this article, $R$ denotes an associative ring with unity. Given a ring $R$, $R[x]$ is the ring of polynomials over $R$ in indeterminate $x$. For any polynomial $f(x)\in R[x]$, $C_{f(x)}$ denotes the set of all coefficients of $f(x)$. $M_{n}(R)$ and $U_{n}(R)$ denote the $n\times n$ full matrix ring and upper triangular matrix ring over $R$, respectively. $D_{n}(R)$ is the ring of $n\times n$ upper triangular matrices over $R$ whose diagonal entries are equal. We use $e_{ij}$ for the matrix with $(i, j)th$ entry $1$ and $0$ elsewhere. The symbol $\mathbb{Z}$ and $\mathbb{Q}$ denote ring of integers and ring of rational numbers respectively.\\
Here, $N(R)$ denotes the set of all nilpotent elements of the ring $R$. An element $a \in R$ is strongly nilpotent if every sequence $a_{1}, a_{2}, a_{3}\ldots$ such that $a_{1} = a$ and $a_{n+1}\in a_{n}Ra_{n}$ (for all $n$) is eventually zero, i.e. there exists a positive integer $n$ such that $a_{n} = 0$. Recall that the prime radical (lower nil radical) of a ring $R$ is the intersection of all prime ideals of $R$ and it is denoted by $N_{*}(R)$. So, $N_{*}(R)$ is precisely the collection of all strongly nilpotent elements of $R$, i.e., $N_{*}(R) = \{x\in R : RxR~ is~ nilpotent\}$. It is known that $N_{*}(R)\subseteq N(R)$.\\

A ring $R$ is said to be reduced if it has no nonzero nilpotent elements. In 1974, Armendariz [\cite{E}, Lemma 1] had pointed out reduced ring satisfies the certain condition which is later called Armendariz ring by Rege and Chhawchharia \cite{M} in 1997. A ring $R$ is said to be an Armendariz if whenever two polynomials $f(x), g(x)\in R[x]$ such that $f(x)g(x) = 0$, then $ab = 0$ for each $a \in C_{f(x)}$ and $b \in C_{g(x)}$.\\
Recall that a ring $R$ is said to be reversible if $ab = 0$ implies $ba = 0$ for each $a, b \in R$. A ring $R$ is said to be semicommutative if $ab = 0$ implies $aRb = 0$ for each $a, b \in R$ \cite{G}. Therefore, a reversible ring is semicommutative but converse is not true.\\
In \cite{S}, some results on almost Armendariz rings are given. We define an almost Armendariz ring as follows:\\
 A ring $R$ is said to be an almost Armendariz ring if whenever two polynomials $f(x)$ and $g(x)\in R[x]$ such that $f(x)g(x) = 0$, then $ab \in N_{*}(R)$ for each $a \in C_{f(x)}$ and $b \in C_{g(x)}$. \\
 Clearly, every semicommutative ring is an almost Armendariz ring by \cite{S}. Therefore, an almost Armendariz ring is a generalisation of an Armendariz ring and semicommutative ring.\\

Recently, in few manuscripts, the Armendariz property of a ring were extended and studied over the skew polynomial rings \cite{C, CC}. For an endomorphism $\alpha$ of a ring $R$, the skew polynomial ring $R[x; \alpha]$ consists of the polynomials in $x$ with multiplication subject to the relation $xr = \alpha(r)x$ for each $r \in R$. Due to \cite{CC}, a ring $R$ is an $\alpha$-Armendariz ring if whenever $f(x), g(x) \in R[x; \alpha]$ such that $f(x)g(x) = 0$, then $a_{i}b_{j} = 0$ for $0 \leq i\leq m$ and $0 \leq j \leq n$.\\
Hong et al.\cite{C} generalized the concept of Armendariz ring with respect to an endomorphism $\alpha$ of $R$ and named as  $\alpha$-skew Armendariz ring.\\
A ring $R$ is said to be an $\alpha$-skew Armendariz ring if $f(x) = a_{0}+a_{1}x+a_{2}x^{2}+\cdots+a_{n}x^{m}, g(x) = b_{0}+b_{1}x+b_{2}x^{2}+\cdots+b_{n}x^{n} \in R[x; \alpha]$ such that $f(x)g(x)= 0$, then $a_{i}\alpha^{i}(b_{j}) = 0$ for each $i, j$. \\
They proved that $\alpha$-rigid rings are $\alpha$-skew Armendariz ring. Moreover, if $\alpha^{t} = I$ for some positive integer $t$, then $R$ is $\alpha$-skew Armendariz ring if and only if $R[x]$ is $\overline{\alpha}$-skew Armendariz ring.\\

According to Hashemi and Moussavi \cite{EE}, a ring $R$ is said to be an $\alpha$-compatible if for each $a, b \in R$, $ab = 0 \Leftrightarrow a\alpha(b) = 0$. According to Krempa \cite{J}, an endomorphism $\alpha$ of a ring $R$ is said to be rigid if $a \alpha (a) = 0$ implies $a = 0$ for each $a \in R$. A ring $R$ is $\alpha$-rigid if there exists a rigid endomorphism $\alpha$ of $R$. In 2005, Hashemi and Moussavi \cite{EE}, considered a ring $R$ as $\alpha$-rigid if and only if $R$ is $\alpha$-compatible and reduced. Moreover, by Proposition (3) of \cite{C}, if $R$ is $\alpha$-rigid, then $R[x; \alpha]$ is reduced.\\ Due to \cite{T}, a ring $R$ is $\alpha(*)$ rigid if $a\alpha(a) \in N_{*}(R)$ implies $a \in N_{*}(R)$, where $\alpha$ is an endomorphism of the ring $R$.\\
Now, we are introducing the notions of $\alpha$-almost Armendariz ring and $\alpha$-skew almost Armendariz ring, where $\alpha$ is an endomorphism on given ring $R$. Some results based on these two notions have been discussed here (in Section 2 and Section 3). \\

\section{$\alpha$-almost Armendariz ring}
\begin{df}
Let $\alpha$ be an endomorphism of a ring $R$. Then the ring $R$ is said to be an $\alpha$-almost Armendariz ring if for $f(x) = a_{0}+a_{1}x+a_{2}x^{2}+\ldots+a_{m}x^{m}, g(x) = b_{0}+b_{1}x+b_{2}x^{2}+\ldots+b_{n}x^{n} \in R[x; \alpha]$ such that $f(x)g(x) = 0$, then $a_{i}b_{j}\in N_{*}(R)$ for each $i, j$, where $0 \leq i \leq m$ and $0 \leq j \leq n$.
\end{df}
It is clear, for $\alpha = id_{R}$, almost Armendariz and $\alpha$-almost Armendariz ring are same, where $id_{R}$ is the identity endomorphism of $R$.\\
\begin{rem} For an endomorphism $\alpha$ of a ring $R$, we have the following:
\begin{itemize}
\item[$(1)$]  If $R$ is an $\alpha$-Armendariz ring, then $R$ is an $\alpha$-almost Armendariz ring.
\item[$(2)$]  Every subring $S$ of an $\alpha$-almost Armendariz ring $R$ with $\alpha(S) \subseteq S$ is also an $\alpha$-almost Armendariz subring.
\end{itemize}
\end{rem}
The following example shows that there exists an endomorphism $\alpha$ of an almost Armendariz ring $R$ such that $R$ is not an $\alpha$-almost Armendariz ring.

\begin{ex}
Let $R = R_{1} \oplus R_{2}$, where $R_{1}$, $R_{2}$ be any two reduced rings. Then $R$ is a semicommutative ring and hence $R$ is an almost Armendariz ring. Let $\alpha : R \rightarrow R$ be an endomorphism defined by $\alpha ((a, b)) = (b, a)$. Let $f(x) = (1, 0) + (0, 1)x, g(x) = (0, 1) - (0, 1)x \in R[x; \alpha]$. Then $f(x)g(x) = 0$, but $(0, 1)(0, 1)\notin N_{*}(R)$. Therefore, $R$ is not an $\alpha$-almost Armendariz ring.
\end{ex}

Also, every $\alpha$-rigid ring is an $\alpha$-almost Armendariz ring but converse is not true. In this regard, we have the following example:
\begin{ex} Let $R = \left(
                    \begin{array}{ccc}
                      \mathbb{Z} & \mathbb{Z} & \mathbb{Q} \\
                      0 & \mathbb{Z} & \mathbb{Q} \\
                      0 & 0 & \mathbb{Z} \\
                    \end{array}
                  \right)$ be a ring and $\alpha:R \rightarrow R$ defined by\\

~~~~~~~~~~~~~~~~~ $\alpha  \left(\left(
                    \begin{array}{ccc}
                      \mathbb{Z} & \mathbb{Z} & \mathbb{Q} \\
                      0 & \mathbb{Z} & \mathbb{Q} \\
                      0 & 0 & \mathbb{Z} \\
                    \end{array}
                  \right)\right)~ = ~\left(
                    \begin{array}{ccc}
                      \mathbb{Z} & \mathbb{Z} & \mathbb{-Q} \\
                      0 & \mathbb{Z} & \mathbb{-Q} \\
                      0 & 0 & \mathbb{Z} \\
                    \end{array}
                  \right)$.\\

 Then $\alpha$ is an endomorphism on $R$ and hence an $\alpha$-almost Armendariz ring. For this, let $f(x) = \sum _{i=0}^{m}A_{i}x^{i}$ and $g(x) = \sum _{j=0}^{n}B_{j}x^{j} \in R[x]$ such that $f(x)g(x) = 0$, where $A_{i}'s$ and $B_{j}'s$ are\\

$A_{i} = \left(
           \begin{array}{ccc}
             a^{i} & a_{12}^{i} & a_{13}^{i} \\
             0 & a^{i} & a_{23}^{i} \\
             0 & 0 & a^{i} \\
           \end{array}
         \right)$
,~~~~~~~~
     $B_{j} = \left(
           \begin{array}{ccc}
             b^{j} & b_{12}^{j} & b_{13}^{j} \\
             0 & b^{j} & b_{23}^{j} \\
             0 & 0 & b^{j} \\
           \end{array}
         \right)$.\\
Now, from $f(x)g(x) = 0$, we have $\Big(\sum _{i=0}^{m}a^{i}x^{i}\Big)\Big(\sum_{j=0}^{n}b^{j}x^{j}\Big) = 0\in \mathbb{Z}[x]$. Since $\mathbb{Z}$ is an Armendariz ring, $a^{i}b^{j} = 0$ for each $0 \leq i \leq m$, $0 \leq j \leq n$. Therefore, $A_{i}B_{j} \in N_{*}(R)$ for each $i, j$. Thus, $R$ is an $\alpha$-almost Armendariz ring. Here, $e_{13}\alpha (e_{13}) = 0$ but $e_{13} \neq 0$. So, $R$ is not an $\alpha$-rigid ring.

\end{ex}

It is well known that an endomorphism $\alpha$ of a ring $R$ can be extended to an endomorphism $\overline{\alpha}$ on $U_{n}(R)$ by defining as $\overline{\alpha}((a_{ij})_{n \times n}) = (\alpha(a_{ij}))_{n \times n}$. Moreover, we have\\
\begin{equation*}
N_{*}(U_{n}(R)) = \left(
                    \begin{array}{ccc}
                      N_{*}(R) & R & R \\
                      0 & \ddots & R \\
                      0 & 0 & N_{*}(R) \\
                    \end{array}
                  \right).
\end{equation*}
\begin{pro}
Let $\alpha$ be an endomorphism of a ring $R$. Then $R$ is an $\alpha$-almost Armendariz ring if and only if for a positive integer $n$, $U_{n}(R)$ is an
$\overline{\alpha}$-almost Armendariz ring.
\end{pro}
\begin{proof} Since subring of an $\alpha$-almost Armendariz ring is an $\alpha$-almost Armendariz.
 Note that\\ $U_{n}(R)[x, \overline{\alpha}] \cong U_{n}(R[x; \alpha])$. So, we prove only necessary part. Let $f(x) = A_{0}+A_{1}x+A_{2}x^{2}+\ldots +A_{p}x^{p}$ and $g(x) = B_{0}+B_{1}x+B_{2}x^{2}+\ldots +B_{q}x^{q} \in U_{n}(R)[x; \overline{\alpha}]$ such that  $f(x)g(x) = 0$, where

$A_{i} = \left(
           \begin{array}{cccc}
             a_{11}^{i} & a_{12}^{i} & \ldots & a_{1n}^{i} \\
             0 & a_{22}^{i} & \ldots & a_{nn}^{i} \\
             \vdots & \vdots & \ddots & \vdots \\
             0 & 0 & \ldots & a_{nn}^{i} \\
           \end{array}
         \right)$,
         $B_{j} = \left(
           \begin{array}{cccc}
             b_{11}^{j} & b_{12}^{j} & \ldots & b_{1n}^{j} \\
             0 & b_{22}^{j} & \ldots & b_{nn}^{i} \\
             \vdots & \vdots & \ddots & \vdots \\
             0 & 0 & \ldots & b_{nn}^{j} \\
           \end{array}
         \right)$, for each $0 \leq i \leq p$ and $0 \leq j \leq q$.

Then $f_{r}(x) = \sum _{i=0}^{p}a_{rr}^{i}x^{i}$ and $g_{r}(x) = \sum _{j=0}^{q}a_{rr}^{j}x^{j} \in R[x; \alpha]$ and $f_{r}(x)g_{r}(x) = 0$, for each $1 \leq r \leq n$. Since $R$ is an $\alpha$-almost Armendariz ring, so $a_{rr}^{i}b_{rr}^{j} \in N_{*}(R)$ for each $1 \leq r \leq n$ and each $i, j$. Therefore, $A_{i}B_{j} \in N_{*}(R)$ for each $0 \leq i \leq p$ and $0 \leq j \leq q$. Hence $U_{n}(R)$ is an $\overline{\alpha}$-almost Armendariz ring.
\end{proof}
\begin{cor} A ring $R$ is an almost Armendariz ring if and only if for a positive integer $n$, $U_{n}(R)$ is an almost Armendariz ring.
\end{cor}
\begin{pro} A ring $R$ is an $\alpha$-almost Armendariz if and only if $R[x]/ \langle x^{n}\rangle$ is an $\overline{\alpha}$-almost Armendariz.
\end{pro}
\begin{proof} For $n \geq 2$, $$R[x]/\langle x ^{n}\rangle \cong  \left \{\left(
                                                           \begin{array}{ccccc}
                                                             a_{0} & a_{1} & a_{2} & \ldots & a_{n-1} \\
                                                             0 & a_{0} & a_{1} & \ldots & a_{n-2} \\
                                                             0 & 0 & a_{0} & \ldots & a_{n-3} \\
                                                             \vdots & \vdots & \vdots & \ddots & \vdots\\
                                                             0 & 0 & 0 & \ldots & a_{0} \\
                                                           \end{array}
                                                         \right)
                                              : a_{i} \in R, ~i = 0, ~1,~ 2,\ldots n-1 \right\}.$$
Hence, $R[x]/\langle x ^{n}\rangle$ is an $\alpha$-almost Armendariz by Proposition (2.1). Converse is also true, being subring of an  $\alpha$-almost Armendariz ring is an $\alpha$-almost Armendariz.
\end{proof}
Given a ring $R$ and a bimodule $_{R}M_{R}$, the trivial extension of $R$ by $M$ is the ring $T(R, M)$ with the usual addition and multiplication defined as
$$(r_{1}, m_{1})(r_{2}, m_{2}) = (r_{1}r_{2}, r_{1}m_{2}+m_{1}r_{2}).$$
This is isomorphic to the ring of all matrices of the form $\left(
                                                 \begin{array}{cc}
                                                   r & m \\
                                                   0 & r \\
                                                 \end{array}
                                               \right)$ with usual addition and multiplication of  matrices, where $r\in R$ and $m\in M$.
\begin{cor} Let $\alpha$ be an endomorphism of a ring $R$. Then $R$ is an $\alpha$-almost Armendariz ring if and only if $U(R, R)$ is an $\overline{\alpha}$-almost Armendariz ring.
\end{cor}
\begin{proof}
Since $R[x]/<x^{2}>~\cong~U(R, R)$. Therefore by Proposition (2.2) $U(R, R)$ is an $\overline{\alpha}$-almost Armendariz ring.
\end{proof}
\begin{lem}[\cite{EE}, Lemma 3.2] Let $R$ be an $\alpha$-compatible ring. Then following hold:
\begin{itemize}
\item[$(1)$] If $ab = 0$, then $a\alpha^{m}(b) = \alpha^{m}(a)b = 0$, for all positive integer $m$.
\item[$(2)$]  If $\alpha^{n}(a)b = 0$ for some positive integer $n$, then $ab = 0$.
\end{itemize}
\end{lem}
\begin{lem} Let $R$ be an $\alpha$-compatible ring. Then we have the following:
\begin{itemize}
\item[$(1)$] If $ab \in N_{*}(R)$, then $a\alpha^{k}(b)\in N_{*}(R)$ and $\alpha^{k}(a)b \in N_{*}(R)$, for all positive integer $k$.
\item[$(2)$]If $\alpha^{m}(a)b \in N_{*}(R)$ or $a\alpha^{m}(b) \in N_{*}(R)$ for some positive integer $m$, then $ab \in N_{*}(R)$.
\end{itemize}
\end{lem}
\begin{proof}
$(1)$ Let $ab \in N_{*}(R)$. This means $RabR$ is nilpotent. If $r_{1}abr_{2} \in RabR$ for some $r_{1}, r_{2} \in R$. Then there exist a positive integer $n$ such that $(r_{1}abr_{2})^{n} = 0$.\\
Now, $(r_{1}abr_{2})(r_{1}abr_{2})\ldots(r_{1}abr_{2}) = 0$ implies $(r_{1}abr_{2})(r_{1}abr_{2})(r_{1}abr_{2})\ldots r_{1}a \alpha^{k}(br_{2}) = 0$, since $R$ is $\alpha$-compatible ring.
 Also, $\alpha$ is an endomorphism, therefore
$(r_{1}abr_{2})(r_{1}abr_{2})(r_{1}abr_{2})\\ \ldots r_{1}a \alpha^{k}(b)\alpha^{k}(r_{2}) = 0$. Again, being $R$ an $\alpha$-compatible ring, we have, $(r_{1}abr_{2})(r_{1}abr_{2})(r_{1}abr_{2})\\ \ldots (r_{1}a \alpha^{k}(b)r_{2}) = 0$. This implies
$(r_{1}abr_{2})(r_{1}abr_{2})\ldots r_{1}a\alpha^{k}(br_{2}(r_{1}a\alpha^{k}(b)r_{2}) = 0$ and hence
$(r_{1}abr_{2})(r_{1}abr_{2})\ldots r_{1}a\alpha^{k}(b)\alpha^{k}(r_{2}(r_{1}a\alpha^{k}(b)r_{2}) = 0$. Therefore,
$(r_{1}abr_{2})(r_{1}abr_{2})\ldots r_{1}a\alpha^{k}(b)\\(r_{2}(r_{1}a\alpha^{k}(b)r_{2}) = 0$, due to $\alpha$-compatible ring and hence
$(r_{1}abr_{2})(r_{1}abr_{2})\ldots (r_{1}a\alpha^{k}(b)r_{2})\\(r_{1}a\alpha^{k}(b)r_{2}) = 0$. Continuing this process, we get
\begin{equation*}
(r_{1}a\alpha^{k}(b)r_{2})(r_{1}a\alpha^{k}(b)r_{2})\ldots(r_{1}a\alpha^{k}(b)r_{2}) = 0.
\end{equation*}
This shows that $(r_{1}a\alpha^{m}(b)r_{2})^{n} = 0$. Thus, $a\alpha^{k}(b) \in N_{*}(R)$. Similarly, we prove that $a\alpha^{k}(b) \in N_{*}(R)$.\\

$(2)$ Let $a\alpha^{m}(b) \in N_{*}(R)$. Then $r_{1}a\alpha^{m}(b)r_{2}$ is nilpotent for any $r_{1}, r_{2} \in R$. So, there exist a positive integer $t$ such that $(r_{1}a\alpha^{m}(b)r_{2})^{t} = 0$.\\
Therefore, $(r_{1}a\alpha^{m}(b)r_{2})(r_{1}a\alpha^{m}(b)r_{2})\ldots (r_{1}a\alpha^{m}(b)r_{2}) = 0$ implies
$(r_{1}a\alpha^{m}(b)\alpha^{m}(r_{2})(r_{1}a\alpha^{m}(b)r_{2})\\\ldots (r_{1}a\alpha^{m}(b)r_{2}) = 0$ and hence
$(r_{1}a\alpha^{m}(br_{2}(r_{1}a\alpha^{m}(b)r_{2})\ldots (r_{1}a\alpha^{m}(b)r_{2}) = 0$. This shows
$(r_{1}abr_{2}(r_{1}a\alpha^{m}(b)r_{2})\ldots (r_{1}a\alpha^{m}(b)r_{2}) = 0$. Therefore,
$(r_{1}abr_{2})(r_{1}a\alpha^{m}(b)\alpha^{m}(r_{2}(r_{1}a\alpha^{m}(b)r_{2})\\ \ldots (r_{1}a\alpha^{m}(b)r_{2}) = 0$. Again,
$(r_{1}abr_{2})(r_{1}a\alpha^{m}(br_{2}(r_{1}a\alpha^{m}(b)r_{2})\ldots (r_{1}a\alpha^{m}(b)r_{2}) = 0$. This implies
$(r_{1}abr_{2})(r_{1}abr_{2})(r_{1}a\alpha^{m}(b)r_{2})\ldots (r_{1}a\alpha^{m}(b)r_{2}) = 0$.
Continuing this process, we get $(r_{1}abr_{2})^{t} = 0$ and hence $ab \in N_{*}(R)$.\\
By same procedure we can proof other part.
\end{proof}
\begin{lem} Let $R$ be a semicommutative $\alpha$-compatible ring. Then
\begin{itemize}
\item[$(1)$]  $ab \in N_{*}(R) \Leftrightarrow a\alpha(b) \in N_{*}(R)$.
\item[$(2)$]  $a\alpha(a) \in N_{*}(R) \Rightarrow a \in N_{*}(R)$.
\end{itemize}
\end{lem}
\begin{proof} $(1)$ It follows by Lemma $(2.2)$.\\
$(2)$ Let $a\alpha(a) \in N_{*}(R)$. Then, from $(1)$, $a^{2} \in N_{*}(R)$ and being $R$ semicommutative, $a \in N_{*}(R)$.
\end{proof}
\begin{rem} Every $\alpha$-compatible and semicommutative ring is an $\alpha(*)$ ring.
\end{rem}
\begin{pro} \cite{CCC} If $R$ is $\alpha(*)$ rigid ring, then $N_{*}(R[x; \alpha]) \subseteq N_{*}(R)[x; \alpha]$.
\end{pro}

\begin{pro} Let $R$ be an $\alpha$-almost Armendariz ring. For $a, b \in R$, we have the following :
\begin{itemize}
\item[$(1)$] If $ab = 0$, then $a\alpha(b) \in N_{*}(R)$.
\item[$(2)$] If $a\alpha^{m}(b) = 0$ for some positive integer $m$, then $ab \in N_{*}(R)$.
\end{itemize}
\end{pro}
\begin{proof}
\begin{itemize}
\item[$(1)$] Let $ab = 0$, $a, b \in R$. Assume $p(x) = \alpha(a)x \in R[x, \alpha]$, $q(x) = bx \in R[x; \alpha]$, then $p(x)q(x) = \alpha(a)x bx = \alpha(a)\alpha(b)x^{2} = \alpha(ab)x^{2} = 0$. Therefore, $\alpha(a)b \in N_{*}(R)$, since $R$ is an $\alpha$-almost Armendariz ring.
\item[$(2)$] Let $a\alpha^{m}(b) = 0$ for some positive integer $m \geq 1$. Suppose $p(x) = ax^{m}$, $q(x) = bx \in R[x; \alpha]$, therefore $p(x)q(x) = ax^{m}bx = a\alpha^{m}(b)x^{2} = 0$. Hence $ab \in N_{*}(R)$, since $R$ is an $\alpha$-almost Armendariz ring.
\end{itemize}
\end{proof}

\begin{thm} Let $R$ be a semicommutative $\alpha$-compatible ring. If $R[x; \alpha]$ is an almost Armendariz ring, then $R$ is an $\alpha$-almost Armendariz ring.
\end{thm}
\begin{proof} Let $R[x; \alpha]$ be an almost Armendariz ring and $f(x) = \sum_{i = 0}^{m}a_{i}x^{i}, g(x) = \sum_{j = 0}^{n}b_{j}x^{j} \in R[x; \alpha]$ such that $f(x)g(x) = 0$. Then $s(y)t(y) = 0$, where $s(y) = a_{0}+(a_{1}x)y+(a_{2}x^{2})y^{2}+\cdots+(a_{m}x^{m})y^{m}$ and $t(y) = b_{0}+(b_{1}x)y+(b_{2}x^{2})y^{2}+\cdots+(b_{n}x^{n})y^{n} \in (R[x; \alpha])[y]$. Since $R[x; \alpha]$ is an almost Armendariz ring, $a_{i}x^{i}b_{j}x^{j} \in N_{*}(R[x; \alpha])$ for each $i, j$. Also by Proposition (2.3), $a_{i}\alpha^{i}(b_{j}) \in N_{*}(R)$ for each $i, j$. Finally, by Lemma $(2.2)$, $a_{i}b_{j} \in N_{*}(R)$.
\end{proof}

\begin{pro} An $\alpha$-compatible semicommutative ring is an $\alpha$-almost Armendariz ring.
\end{pro}

\begin{proof} Let $f(x) = \sum_{i=0}^{m}a_{i}x^{i}$ and $g(x) = \sum_{j=0}^{n}b_{j}x^{j} \in R[x; \alpha]$ such that \\$f(x)g(x) =  (\sum_{i=0}^{m}a_{i}x^{i})(\sum_{j=0}^{n}b_{j}x^{j}) = \sum_{l=0}^{m+n} (\sum_{i+j = l} a_{i}\alpha^{i}(b_{j}))x^{l} = 0$. Then\\
\begin{equation*}
\sum_{i+j = l} a_{i}\alpha^{i}(b_{j}) = 0, ~~~~~~~~~~~~~~~~~~    l = 0,1, 2, \ldots, m+n.
\end{equation*}
To prove $a_{i}b_{j} \in N_{*}(R)$, we use induction on $i+j$. If $i+j = 0$, then $a_{0}b_{0} = 0$, therefore $a_{0}b_{0} \in N_{*}(R)$. Now, assume that result is true for $i+j<l$, where $l$ is a positive integer, i.e. $a_{i}b_{j} \in N_{*}(R)$, for $i+j<l$. To prove $a_{i}b_{j} \in N_{*}(R)$, for $i+j = l$, we use the coefficient of $x^{l}$ in product of $f(x)g(x)$, which is \\
\begin{equation}
a_{0}b_{l}+a_{1}\alpha(b_{l-1})+a_{2}\alpha^{2}(b_{l-2})+\ldots+a_{l}\alpha^{l}(b_{0}) = 0
\end{equation}
Multiplying by $b_{0}$ from left in equation $(2.1)$, we have
\begin{equation*}
b_{0}a_{l}\alpha^{l}(b_{0}) = -(b_{0}a_{0}b_{l}+b_{0}a_{1}\alpha(b_{l-1})+b_{0}a_{2}\alpha^{2}(b_{l-2})+\ldots+b_{0}a_{l-1}\alpha^{l-1}(b_{1})).
\end{equation*}
Since $a_{i}b_{0} \in N_{*}(R)$ for each $i = 0, 1, 2, \ldots, (l-1)$ and $b_{0}a_{i} \in N_{*}(R)$ for each $i = 0, 1, 2, \ldots, (l-1)$. Therefore $b_{0}a_{l}\alpha^{l}(b_{0})\in N_{*}(R)$. Hence, $b_{0}a_{l}\alpha^{l}(b_{0})\alpha^{l}(a_{l})\in N_{*}(R)$ this implies $b_{0}a_{l}\alpha^{l}(b_{0}a_{l})\in N_{*}(R)$. Therefore, by Lemma(2.2), $(b_{0}a_{l})^{2} \in N_{*}(R)$ and hence $(b_{0}a_{l}) \in N_{*}(R)$. Therefore, $a_{l}b_{0} \in N_{*}(R)$. Again, multiplying by $b_{1}$ in $(2.1)$, we have\\
\begin{equation*}
b_{1}a_{l-1}\alpha^{l-1}(b_{1}) = -(b_{1}a_{0}b_{l}+b_{1}a_{1}\alpha(b_{l-1})+b_{1}a_{2}\alpha^{2}(b_{l-2})+\ldots+b_{1}a_{l}\alpha^{l}(b_{0})) \in N_{*}(R)
\end{equation*}
$b_{1}a_{l-1}\alpha^{l-1}(b_{1}) \in N_{*}(R)$, so $b_{1}a_{l-1}\alpha^{l-1}(b_{1}a_{l-1}) \in N_{*}(R)$. By above Lemma (2.2), $b_{1}a_{l-1} \in N_{*}(R)$. Continuing this process, we have $a_{0}b_{l}, a_{1}b_{l-1}, a_{2}b_{l-2}, \ldots, a_{l}b_{0} \in N_{*}(R)$. Therefore, $a_{i}b_{j} \in N_{*}(R)$ for each $i, j$ where $0 \leq i \leq m$ and $0 \leq j \leq n$. Thus, $R$ is $\alpha$-almost Armendariz ring.
\end{proof}
Also by following the extension of an endomorphism $\alpha$ of a ring $R$ to its ring of polynomials $R[x]$, given by Hong et al. in \cite{CC}, we have the following result:
\begin{pro} Let $\alpha$ be an endomorphism of a ring $R$ and $\alpha^{k} = I_{R}$ for some positive integer $k$. Then $R$ is an $\alpha$-almost Armendariz ring if and only if $R[x]$ is an $\alpha$-almost Armendariz ring.
\end{pro}
\begin{proof}
 Let $p(y) = f_{0}(x) + f_{1}(x)y + \cdots + f_{m}(x)y^{m}$, $q(y) = g_{0} + g_{1}(x)y +\cdots +g_{n}y^{n} \in R[x][y; \alpha]$ such that $p(y)q(y) = 0$, where $f_{i}(x), g_{j}(x) \in R[x]$. Write $f_{i}(x) = a_{i0} + a_{i1}x + \cdots + a_{is_{i}}x^{s_{i}}$, $g_{j}(x) = b_{j0} + b_{j1}x + \cdots + b_{jt_{j}}x^{t_{j}}$, for each $0 \leq i \leq m$ and $0 \leq j \leq n$, where $a_{i0}, a_{i1}, \ldots, a_{is_{i}}, b_{j0}, b_{j1}, \ldots, b_{jt_{j}} \in R$. We have to show $f_{i}(x)g_{j}(x) \in N_{*}(R[x])$, for each $0 \leq i \leq m$ and $0 \leq j \leq n$. Choose a positive integer $l$ such that $l > deg(f_{0}(x)) + deg(f_{1}(x)) + \cdots + deg(f_{m}(x)) + deg(g_{0}(x)) + deg (g_{1}(x)) + \cdots + deg(g_{n}(x))$.
Now, put
\begin{eqnarray*}
p(x^{lt+1})&=&f(x) = f_{0}(x) + f_{1}(x)x^{lt+1} + f_{2}x^{2lt+2} + \cdots + f_{m}(x)x^{mlt+m};\\
q(x^{lt+1})&=&g(x) = g_{0}(x) + g_{1}(x)x^{lt+1} + g_{2}x^{2lt+2} + \cdots + g_{n}x^{nlt+n}.
\end{eqnarray*}
Then $p(x^{lt+1}), q(x^{lt+1}) \in R[x]$ and coefficients of $p(x^{lt+1})$ and $q(x^{lt+1})$ are equal to the sets of coefficients of $f_{i}$ and $g_{j}$ respectively. Since $p(y)q(y) = 0 \in R[x][y; \alpha]$ and $x$ commutes with the elements of $R$ in the polynomials of $R[x]$ and $\alpha^{k} = I_{R}$, we have $p(x^{lt+1})q(x^{lt+1}) = 0 \in R[x; \alpha]$. Since $R$ is an $\alpha$-almost Armendariz ring, we have $a_{ic}b_{jd} \in N_{*}(R),$ for all $0 \leq i \leq m$, $0 \leq j \leq n$, $c \in \{0, 1, \ldots, s_{i}\}$ and $d \in \{0, 1, \ldots, t_{j}\}$. Therefore, $f_{i}(x)g_{j}(x) \in N_{*}(R)[x] = N_{*}(R[x]),$ for all $0 \leq i \leq m$ and $0 \leq j \leq n$. Thus, $R[x]$ is an  $\alpha$-almost Armendariz ring.
\end{proof}
\begin{pro}
Let $R$ be an abelian ring with $\alpha (e) = e$ for an idempotent element $e \in R$. Then the following statements are equivalent:
\begin{itemize}
\item[$(1)$] $R$ is an $\alpha$-almost Armendariz ring.
\item[$(2)$] $eR$ and $(1-e)R$ are $\alpha$-almost Armendariz rings.
\end{itemize}
\end{pro}
\begin{proof} $(1) \Rightarrow (2)$ is obvious, since subring of an $\alpha$-almost Armendariz ring is $\alpha$-almost Armendariz ring.\\
$(2) \Rightarrow (1):$ Let $f(x) = \sum _{i=0}^{m}a_{i}x^{i}$, $g(x) = \sum _{j=0}^{n}b_{j}x^{j} \in R[x; \alpha]$ such that $fg = 0$. Then $(ef(x))(eg(x)) = 0$ and $(1-e)f(x) (1-e)g(x) = 0$. Since $eR$ is an $\alpha$-almost Armendariz ring, therefore $ea_{i}b_{j} \in N_{*}(R)$. Similarly, $(1-e)a_{i}b_{j} \in N_{*}(R)$, since $(1-e)R$ is also an $\alpha$-almost Armendariz ring. Therefore, $a_{i}b_{j} \in N_{*}(R)$ for each $i, j,$ where $0 \leq i \leq m$ and $0 \leq j \leq n$. Thus, $R$ is an $\alpha$-almost Armendariz ring.
\end{proof}

\begin{pro} Let $R$ be $\alpha$-compatible ring. If $R$ is $\alpha$-almost Armendariz ring and $a^{2} = 0$, $b^{2} = 0$, then $aba \in N_{*}(R)$ and hence $ab, a+b \in N(R)$.
\end{pro}
\begin{proof} Let $f(x) = a(1-bx)$ and $g(x) = a + b\alpha(a)x$. Then $f(x)g(x) = (a-abx)(a+b\alpha(a)x) = a^{2}+ab\alpha(a)x-ab\alpha(a)x-ab\alpha(b\alpha(a))x^{2} = 0$. Therefore, $aba \in N_{*}(R)$ and hence $ab \in N(R)$. By \cite{S}, $a+b \in N(R)$.
\end{proof}

\section{$\alpha$-skew almost Armendariz ring}
\begin{df} Let $\alpha$ be an endomorphism of a ring $R$. The ring $R$ is said to be an $\alpha$-skew almost Armendariz ring if whenever $f(x) = a_{0}+a_{1}x+a_{2}x^{2}+\cdots+a_{m}x^{m}$, $g(x) = b_{0}+b_{1}x+b_{2}x^{2}+\cdots+b_{n}x^{n} \in R[x; \alpha]$ such that $f(x)g(x) = 0$, then $a_{i}\alpha^{i}(b_{j}) \in N_{*}(R)$ for each $i$, $j$, where $0 \leq i \leq m$ and $0 \leq j \leq n$.
\end{df}
It is clear by definition that a subring of an $\alpha$-skew almost Armendariz ring is an $\alpha$-skew almost Armendariz ring.\\
Let $\alpha$ be an endomorphism on a ring $R$ and $M_{n}(R)$ be $n\times n$ full matrix ring over $R$. Let $\overline{\alpha} : M_{n}(R) \rightarrow M_{n}(R)$ defined by $\overline{\alpha}((a_{ij})) = (\alpha(a_{ij}))$. Then $\overline{\alpha}$ is an endomorphism on $M_{n}(R)$ ($U_{n}(R)$).
Also, we know that \\
\begin{equation*}
N_{*}(U_{n}(R)) = \left(
                    \begin{array}{ccc}
                      N_{*}(R) & R & R \\
                      0 & \ddots & R \\
                      0 & 0 & N_{*}(R) \\
                    \end{array}
                  \right).
\end{equation*}
Moreover, by Example 14 of \cite{C}, $R$ is $\alpha$-skew Armendariz ring but $U_{n}(R)$ $(n\geq 2)$ need not be $\overline{\alpha}$-skew Armendariz ring. For an $\alpha$-skew almost Armendariz ring, we have the following:
\begin{pro} Let $\alpha$ be an endomorphism of a ring $R$. Then $R$ is an $\alpha$-skew almost Armendariz ring if and only if for any positive integer $n$, $U_{n}(R)$ is an $\overline{\alpha}$-skew almost Armendariz ring.
\end{pro}
\begin{proof} Since subring of an $\alpha$-skew almost Armendariz ring is an $\alpha$-skew almost Armendariz ring. Therefore, $R$ is an $\alpha$-skew almost Armendariz ring.\\
Conversely, let $f(x) = A_{0}+A_{1}x+A_{2}x^{2}+\cdots +A_{r}x^{r}$ and $g(x) = B_{0}+B_{1}x+B_{2}x^{2}+\cdots +B_{s}x^{s}$ $\in U_{n}(R[x;\alpha])$ such that  $f(x)g(x) = 0$, where\\

$A_{i} = \left(
           \begin{array}{cccc}
             a_{11}^{i} & a_{12}^{i} & \ldots & a_{1n}^{i} \\
             0 & a_{22}^{i} & \ldots & a_{nn}^{i} \\
             \vdots & \vdots & \ddots & \vdots \\
             0 & 0 & \ldots & a_{nn}^{i} \\
           \end{array}
         \right)$,
        and  $B_{j} = \left(
           \begin{array}{cccc}
             b_{11}^{j} & b_{12}^{j} & \ldots & b_{1n}^{j} \\
             0 & b_{22}^{j} & \ldots & b_{nn}^{i} \\
             \vdots & \vdots & \ddots & \vdots \\
             0 & 0 & \ldots & b_{nn}^{j} \\
           \end{array}
         \right)$, for each $0 \leq i \leq r$ and $0 \leq j \leq s$.\\
If $f_{t}(x) = \sum _{i=0}^{r}a_{tt}^{i}x^{i},$ $g_{t}(x) = \sum _{j=0}^{s}a_{tt}^{j}x^{j} \in R[x; \alpha]$, then $f_{t}(x)g_{t}(x) = 0$, for each $1 \leq t \leq n$. Since $R$ is the $\alpha$-skew almost Armendariz ring, therefore $a_{tt}^{i}\alpha ^{i}(b_{tt}^{j}) \in N_{*}(R)$ for each $1 \leq t \leq n$ and each $i, j$. Also, $A_{i}\overline{\alpha}^{i}(B_{j}) \in N_{*}(R)$ for each $0 \leq i \leq r$ and $0 \leq j \leq s$. Thus, $U_{n}(R)$ is an $\overline{\alpha}$-skew almost Armendariz ring.
\end{proof}
\begin{cor} If $R$ is an $\alpha$-skew Armendariz ring, then for any positive integer $n$, $U_{n}(R)$ is an $\overline{\alpha}$-skew almost Armendariz ring.
\end{cor}
It is noted that full matrix ring $M_{n}(R)$ over $R$ need not be an $\overline{\alpha}$-skew almost Armendariz ring.
\begin{ex}
Let $\alpha$ be an endomorphism of the ring $R$. Consider $T = M_{2}(R)$.  Let
\begin{equation*}
f(x) = \left(
          \begin{array}{cc}
            1 & 0 \\
            0 & 0 \\
          \end{array}
        \right)+\left(
                  \begin{array}{cc}
                    0 & 1 \\
                    0 & 0 \\
                  \end{array}
                \right)x
\end{equation*}
\begin{equation*}
                g(x) = \left(
                          \begin{array}{cc}
                            0 & 0 \\
                            1 & 1 \\
                          \end{array}
                        \right)+\left(
                                  \begin{array}{cc}
                                    -1 & -1 \\
                                    0 & 0 \\
                                  \end{array}
                                \right)x \in T[x; \overline{\alpha}].
\end{equation*}
Then $f(x)g(x) = 0$, but $\left(
                                              \begin{array}{cc}
                                                1 & 0 \\
                                                0 & 0 \\
                                              \end{array}
                                            \right)\alpha \Big(\left(
                                                            \begin{array}{cc}
                                                              1 & 1 \\
                                                              0 & 0 \\
                                                            \end{array}
                                                          \right)\Big) = \left(
                                                                      \begin{array}{cc}
                                                                        1 & 1 \\
                                                                        0 & 0 \\
                                                                      \end{array}
                                                                    \right)$ is not a strongly nilpotent element.\\

                                                                     Hence, $T$ is not an $\overline{\alpha}$-skew almost Armendariz ring.

\end{ex}
Recall that, for an endomorphism $\alpha$ of a ring $R$, an ideal $I$ is said to be an $\alpha$-ideal if $\alpha(I)\subseteq I$. For an $\alpha$-ideal, we define $\overline{\alpha} : R/I \rightarrow R/I$ by $\overline{\alpha}(a+I) = \alpha(a)+I$ for $a \in R$. Here, $\overline{\alpha}$ is an endomorphism of the factor ring $\frac{R}{I}$.
\begin{pro} Let $\alpha$ be an endomorphism of a ring $R$ and $I$ be an $\alpha$-ideal. If $R/I$ is an $\overline{\alpha}$-skew almost Armendariz ring with $I\subseteq N_{*}(R)$, then $R$ is an $\alpha$-skew almost Armendariz ring.
\end{pro}
\begin{proof} Let $f(x) = \sum_{i = 0}^{m}a_{i}x^{i},$ $g(x) = \sum_{j = 0}^{n}b_{j}x^{j} \in R[x; \alpha]$ such that $f(x)g(x) = 0$. Then $(\sum_{i = 0}^{m}\overline{a_{i}}x^{i})(\sum_{j = 0}^{n}\overline{b_{j}}x^{j}) = 0$. Therefore, $\overline{a_{i}}\alpha ^{i}(\overline{b_{j}}) \in N_{*}(R/I) = N_{*}(R)/I$ for each $i, j$. This implies $a_{i}\alpha^{i}(b_{j}) \in N_{*}(R)$ for each $i, j$. Thus, $R$ is an $\alpha$-skew almost Armendariz ring.
\end{proof}
\begin{pro} Let $\alpha$ be an endomorphism on an abelian ring $R$ such that $\alpha(e) = e$, for each idempotent element $e \in R$. Then $R$ is an $\alpha$-skew almost Armendariz ring if and only if $eR$ and $(1-e)R$ are $\alpha$-skew almost Armendariz rings.
\end{pro}
\begin{proof} Let $R$ is an $\alpha$-skew almost  Armendariz ring. Since $eR$ and $(1-e)R$ are subrings of $R$, therefore $eR$ and $(1-e)R$ are $\alpha$-skew almost Armendariz rings. \\ Conversely, let $f(x) = \sum_{i = 0}^{m}a_{i}x^{i},$ $g(x) = \sum_{j = 0}^{n}b_{j}x^{j} \in R[x; \alpha]$ such that $f(x)g(x) = 0$. Let $f_{1}(x) = ef(x)$, $f_{2}(x) = (1-e)f(x)$, $g_{1}(x) = eg(x)$ and $g_{2}(x) = (1-e)g(x)$. Then $f_{1}(x)g_{1}(x) = 0$ and $f_{2}(x)g_{2}(x) = 0$ in $R[x; \alpha]$. Since $eR$ and $(1-e)R$ is an $\alpha$-skew almost Armendariz rings and also $N_{*}(eR) = eN_{*}(R)$, $N_{*}((1-e)R) = (1-e)N_{*}(R)$, therefore $ea_{i}\alpha ^{i}(b_{j}) \in N_{*}(R)$ and $(1-e)a_{i}\alpha ^{i}(b_{j}) \in N_{*}(R)$ for each $i, j$. Therefore, $a_{i}\alpha ^{i}(b_{j}) \in N_{*}(R)$ for each $i, j$. Hence, $R$ is an $\alpha$-skew almost Armendariz ring.
\end{proof}

\begin{lem} Let $\alpha$ be an endomorphism on a reversible ring $R$ such that $a\alpha(b) = 0$, whenever $ab = 0$ for any $a, b \in R$. If $ab \in N_{*}(R)$, then $a\alpha ^{t}(b) \in N_{*}(R)$ for any positive integers $t$.
\end{lem}
\begin{proof} Let $ab \in N_{*}(R)$. Then $RabR$ is nilpotent. Therefore, there exist a positive integer $m$ such that for any $r_{1}, r_{2} \in R$, $(r_{1}abr_{2})^{m} = 0$. This implies $(r_{1}abr_{2})^{m-1}(r_{1}abr_{2}) = 0$ and  $((r_{1}abr_{2})^{m-1}(r_{1}ab))r_{2} = 0$. Since $R$ is reversible, $(r_{2}(r_{1}abr_{2})^{m-1})(r_{1}ab) = 0$, and by assumption, there exist a positive integer $t$ such that $(r_{2}(r_{1}abr_{2})^{m-1}r_{1}a)\alpha ^{t}(b) = 0$. Also by repeated application of reversibility of $R$, $(r_{1}abr_{2})^{m-1}(r_{1}a\alpha^{t}(b)r_{2})= 0$ and $(r_{1}a\alpha ^{t}(b)r_{2})(r_{1}abr_{2})^{m-1} \\= 0$. \\
Again, $(r_{1}a\alpha ^{t}(b)r_{2})(r_{1}abr_{2})^{m-2}(r_{1}abr_{2}) = 0$ implies, $r_{2}{(r_{1}a\alpha ^{t}(b)r_{2})(r_{1}abr_{2})^{m-2}r_{1}ab} = 0$ and hence $r_{2}{(r_{1}a\alpha ^{t}(b)r_{2})(r_{1}abr_{2})^{m-2}}r_{1}a\alpha ^{t}(b) = 0$. This implies ${(r_{1}a\alpha ^{t}(b)r_{2})(r_{1}abr_{2})^{m-2}}(r_{1}a\alpha ^{t}(b)r_{2})\\ = 0$. Hence, ${(r_{1}a\alpha ^{t}(b)r_{2})}^{2}(r_{1}abr_{2})^{m-2} = 0$. Continuing this process, we get $(r_{1}a\alpha ^{t}(b)r_{2})^{m} = 0$. Therefore, $Ra\alpha ^{t}(b)R$ is nilpotent for any positive integer $t$. Thus, $a\alpha^{t}(b) \in N_{*}(R)$.
\end{proof}
\begin{pro} Let $R$ be a reversible ring and $\alpha$ be an endomorphism of $R$ such that $a\alpha(b) = 0$, whenever $ab = 0$ for any $a, b \in R$. Then $R$ is an $\alpha$-skew almost Armendariz ring.
\end{pro}

\begin{proof} Let  $f(x) = a_{0}+a_{1}x+a_{2}x^{2}+\cdots+a_{m}x^{m}$, $g(x) = b_{0}+b_{1}x+b_{2}x^{2}+\cdots+b_{n}x^{n} \in R[x; \alpha]$ such that $f(x)g(x) = 0$. Then we have the following equations:

$\begin{array}{ll}
a_{0}b_{0} = 0  ~~~~~~  \hfill (1)\\
a_{0}b_{1}+a_{1}\alpha(b_{0}) = 0  ~~~~~~  \hfill (2)\\
a_{0}b_{2}+a_{1}\alpha(b_{1})+a_{2}\alpha^{2}(b_{0})= 0  ~~~~~~  \hfill (3)\\
a_{0}b_{l}+a_{1}\alpha(b_{l-1})+a_{2}\alpha^{2}(b_{l-2})+a_{l}\alpha^{l}b_{0}= 0 ~~~~~  \hfill (4)\\
\ldots ~~~\ldots~~~ \ldots  \\

a_{m}\alpha^{m}(b_{n}) = 0  ~~~~~  \hfill (5)\\
\end{array}$\\


To prove $a_{i}\alpha^{i}(b_{j}) \in N_{*}(R)$, we use principle of induction on $i+j$.\\
If $i+j = 0$, then $a_{0}b_{0} = 0 \in N_{*}(R)$.\\
Let result is true for $i+j<l$ where $l\leq m+n$, i.e. $a_{i}\alpha^{i}(b_{j}) \in N_{*}(R)$, for $i+j<l$. Now, we prove $a_{i}\alpha^{i}(b_{j}) \in N_{*}(R)$ for $i+j = l$.
\end{proof}
Multiplying equation (4), by $a_{0}$ from left, we have
\begin{equation*}
a_{0}a_{0}b_{l}+a_{0}a_{1}\alpha(b_{l-1})+a_{0}a_{2}\alpha^{2}(b_{l-2})+a_{0}a_{l}\alpha^{l}(b_{0}) = 0.
\end{equation*}
By Lemma (3.1), $a_{i}\alpha ^{l}(b_{0}) \in N_{*}(R)$ for $i<l$ and $a_{i}r\alpha ^{l}(b_{0}) \in N_{*}(R)$ for any $r \in R$, since $R$ is the reversible ring. Therefore, by above equation, we have,
\begin{equation*}
a_{0}a_{0}b_{l} = -(a_{0}a_{1}\alpha(b_{l-1})+a_{0}a_{2}\alpha^{2}(b_{l-2})+a_{0}a_{l}\alpha^{l}(b_{0})) \in N_{*}(R).
\end{equation*}
This implies $a_{0}a_{0}b_{l} \in N_{*}(R)$, again $R$ is reversible so $a_{0}b_{l}a_{0}b_{l} \in N_{*}(R)$, hence $a_{0}b_{l} \in N_{*}(R)$. Also, multiplying equation (4) by $a_{1}$ from left, we get,
 $a_{1}\alpha(b_{l-1}) \in N_{*}(R)$. Continuing this process, we obtain $a_{i}\alpha^{i}(b_{j}) \in N_{*}(R)$, for $i+j = l$.
 Thus, by induction $a_{i}\alpha ^{i}(b_{j})\in N_{*}(R)$ for each $i, j$. Hence, $R$ is an $\alpha$-skew almost Armendariz ring.

\begin{lem}(Lemma 7, \cite{T}) Let $R$ be an $\alpha(*)$-ring with $\alpha$-ideal $N_{*}(R)$. If $ab \in N_{*}(R)$, then $a\alpha^{n}(b) \in N_{*}(R)$ and $\alpha^{n}(a)b \in N_{*}(R)$ for any positive integer $n$. Conversely, if $a\alpha^{k}(b)$ or $\alpha^{k}(a)b \in N_{*}(R)$ for some positive integer $k$, then $ab \in N_{*}(R)$.
\end{lem}
\begin{thm}(Theorem 8, \cite{T}) Let $R$ be an $\alpha(*)$ ring with an $\alpha$-ideal $N_{*}(R)$. Assume that $p(x) = \sum_{i=0}^{m}a_{i}x^{i}$ and $q(x) = \sum_{j=0}^{n}b_{j}x^{j} \in R[x; \alpha]$. Then the following statements are equivalent:
\begin{itemize}
\item[$(1)$] $p(x)q(x) \in N_{*}(R)[x; \alpha]$.
\item[$(2)$] $a_{i}b_{j} \in N_{*}(R)$ for each $i, j$, where $0\leq i \leq m$ and $0\leq j \leq n$.
\end{itemize}
\end{thm}
\begin{rem} If $R$ is an $\alpha(*)$ ring with an $\alpha$-ideal $N_{*}(R)$. Then $R$ is an $\alpha$-skew almost Armendariz ring.
\end{rem}


\begin{thm} Let $R$ be a reversible ring and $\alpha$ be an endomorphism of $R$ such that $a\alpha(b) = 0$, whenever $ab = 0$ for any $a, b \in R$. If for some positive integer $k$, $\alpha^{k} = I$, then $R[x]$ is an $\overline{\alpha}$-skew almost Armendariz ring.
\end{thm}
\begin{proof}
Let $p(y) = f_{0}(x) + f_{1}(x)y + \cdots + f_{m}(x)y^{m}$, $q(y) = g_{0} + g_{1}(x)y +\cdots +g_{n}y^{n}\in R[x][y; \overline{\alpha}]$ such that $p(y)q(y) = 0$, where $f_{i}(x), g_{j}(x) \in R[x]$. Here, $f_{i}(x) = a_{i0} + a_{i1}x + \cdots + a_{is_{i}}x^{s_{i}}$, $g_{j}(x) = b_{j0} + b_{j1}x + \cdots + b_{jt_{j}}x^{t_{j}}$, for each $0 \leq i \leq m$ and $0 \leq j \leq n$, where $a_{i0}, a_{i1}, \ldots, a_{is_{i}}, b_{j0}, b_{j1}, \ldots, b_{jt_{j}} \in R$. We have to prove $f_{i}(x)\overline{\alpha}^{i}(g_{j}(x)) \in N_{*}(R[x])$, for each $0 \leq i \leq m$ and $0 \leq j \leq n$.\\ Choose a positive integer $v$ such that $v > deg(f_{0}(x)) + deg(f_{1}(x)) + \cdots + deg(f_{m}(x)) + deg(g_{0}(x)) + deg (g_{1}(x)) + \cdots + deg(g_{n}(x))$.
Now,
\begin{eqnarray*}
p(x^{kv+1})&=&f(x) = f_{0}(x) + f_{1}(x)x^{kv+1} + f_{2}(x)x^{2kv+2} + \cdots + f_{m}(x)x^{mkv+m};\\
q(x^{kv+1})&=&g(x) = g_{0}(x) + g_{1}(x)x^{kv+1} + g_{2}(x)x^{2kv+2} + \cdots + g_{n}(x)x^{nkv+n}.
\end{eqnarray*}
Then $p(x^{kv+1}), q(x^{kv+1})\in R[x]$ and sets of coefficients of $p(x^{kv+1})$ and $q(x^{kv+1})$ are equal to the sets of coefficients of $f_{i}'s$ and $g_{j}'s$ respectively. Since $p(y)q(y) = 0 \in R[x][y; \alpha]$ and $x$ commutes with elements of $R$ and $\alpha^{k} = I_{R}$, we have $p(x^{kv+1})q(x^{kv+1}) = 0 \in R[x; \alpha]$. Since, $R$ is an $\alpha$-skew almost Armendariz ring, therefore $a_{ic}\alpha ^{i}(b_{jd}) \in N_{*}(R)$, for all $0 \leq i \leq m$, $0 \leq j \leq n$, $c \in \{0, 1, \ldots, s_{i}\}$ and $d \in \{0, 1, \ldots, t_{j}\}$. Hence, $f_{i}(x)\overline{\alpha}^{i}(g_{j}(x)) \in N_{*}(R)[x] = N_{*}(R[x])$, for all $0 \leq i \leq m$ and $0 \leq j \leq n$. Thus, $R[x]$ is an $\overline{\alpha}$-skew almost Armendariz ring.
\end{proof}
\begin{thm} Let $R$ be a reversible ring and $\alpha$ be an endomorphism of $R$ such that $a\alpha(b) = 0$, whenever $ab = 0$ for any $a, b \in R$. If for some positive integer $k$, $\alpha^{k} = I$, then $R[x; \alpha]$ is an almost Armendariz ring.
\end{thm}
\begin{proof} Let $p(y) = f_{0}(x) + f_{1}(x)y + \cdots + f_{m}(x)y^{m}$, $q(y) = g_{0} + g_{1}(x)y +\cdots +g_{n}y^{n}\in R[x; \alpha][y]$ such that $p(y)q(y) = 0$, where $f_{i}(x), g_{j}(x) \in R[x; \alpha]$. Write $f_{i}(x) = a_{i0} + a_{i1}x + \cdots + a_{is_{i}}x^{s_{i}}$, $g_{j}(x) = b_{j0} + b_{j1}x + \cdots + b_{jt_{j}}x^{t_{j}}$, for each $0 \leq i \leq m$ and $0 \leq j \leq n$, where $a_{i0}, a_{i1}, \ldots, a_{is_{i}}, b_{j0}, b_{j1}, \ldots, b_{jt_{j}} \in R$. To prove $f_{i}(x)(g_{j}(x)) \in N_{*}(R[x; \alpha])$, for each $0 \leq i \leq m$ and $0 \leq j \leq n$. Choose a positive integer $w$ such that $w > deg(f_{0}(x)) + deg(f_{1}(x)) + \cdots + deg(f_{m}(x)) + deg(g_{0}(x)) + deg (g_{1}(x)) + \cdots + deg(g_{n}(x))$.
Now,
\begin{eqnarray*}
p(x^{kw}) = f_{0}(x) + f_{1}(x)x^{kw} + f_{2}(x)x^{2kw} + \cdots + f_{m}(x)x^{mkw};\\
q(x^{kw}) = g_{0}(x) + g_{1}(x)x^{kw} + g_{2}(x)x^{2kw} + \cdots + g_{n}(x)x^{nkw}.
\end{eqnarray*}
 Then $p(x^{kw}), q(x^{kw}) \in R[x; \alpha]$. Also coefficients of $p(x^{kw})$ and $q(x^{kw})$ are ultimately the coefficients of $f_{i}'s$ and $g_{j}'s$ respectively. Here, $p(x^{kw})q(x^{kw}) = 0 \in R[x; \alpha]$ and $\alpha^{k} = I$. Since $R$ is $\alpha$-skew almost Armendariz ring by Proposition (3.4), therefore, $a_{ic}\alpha ^{i}(b_{jd}) \in N_{*}(R),$ for all $0 \leq i \leq m$, $0 \leq j \leq n$, $c \in \{0, 1, \ldots, s_{i}\}$ and $d \in \{0, 1, \ldots, t_{j}\}$. Hence, $f_{i}g_{j} \in N_{*}(R[x; \alpha])$ for each $0 \leq i \leq m$ and $0 \leq j \leq n$. Thus, $R[x; \alpha]$ is an almost Armendariz ring.
\end{proof}
\begin{thm} Let $\alpha$ be an endomorphism of $R$ and $\alpha^{t} = I$ for some positive integer $t$. Then $R$ is an $\alpha$-skew almost Armendariz ring if and only if $R[x]$ is an $\alpha$-skew almost Armendariz ring.
\end{thm}
\begin{proof} Let $R$ be an ${\alpha}$-skew almost Armendariz ring. Let $f(y) = p_{0}(x)+p_{1}(x)y+p_{2}(x)y^{2}+\cdots+p_{m}(x)y^{m}$, $g(y) = q_{0}(x)+q_{1}(x)y+q_{2}(x)y^{2}+\cdots+q_{n}(x)y^{n}$ in $R[x][y; \alpha]$ such that $f(y)g(y) = 0$. We also take, $p_{i}(x) = a_{i0} + a_{i1}x + \cdots + a_{iu_{i}}x^{u_{i}}$, $q_{j}(x) = b_{j0} + b_{j1}x + \cdots + b_{jv_{j}}x^{v_{j}}$, for each $0 \leq i \leq m$ and $0 \leq j \leq n$, where $a_{i0}, a_{i1}, \ldots, a_{iu_{i}}, b_{j0}, b_{j1}, \ldots, b_{jv_{j}} \in R$.\\ Choose a positive integer $k$ such that $k > deg(p_{0}(x)) + deg(p_{1}(x)) + \cdots + deg(p_{m}(x)) + deg(q_{0}(x)) + deg (q_{1}(x)) + \cdots + deg(q_{n}(x))$. Now, $f(x^{kt}) = p_{0}(x)+p_{1}(x)x^{kt}+p_{2}(x)x^{2kt}+\cdots+p_{m}(x)x^{mkt}$, $g(x^{kt}) = q_{0}(x)+q_{1}(x)x^{kt}+q_{2}(x)x^{2kt}+\cdots+q_{n}(x)x^{nkt} \in R[x]$. Then the sets of coefficients of $p_{i}^{'s}$ and $q_{i}^{'s}$ are equal to the sets of coefficients of $f(x^{kt})$ and $g(x^{kt})$ respectively. Since $f(y)g(y) = 0$ and $x$ commute with element of $R$ in the polynomial $R[x]$, $\alpha ^{kt} = I$, therefore $f(x^{kt})g(x^{kt}) = 0 \in R[x; \alpha]$. Since $R$ is an $\alpha$-skew almost Armendariz ring, therefore $a_{ic}\alpha ^{i}(b_{jd}) \in N_{*}(R),$ for all $0 \leq i \leq m$, $0 \leq j \leq n$, $c \in \{0, 1, \ldots, u_{i}\}$ and $d \in \{0, 1, \ldots, v_{j}\}$. Hence $p_{i}(x)\alpha ^{i}(q_{j}(x)) \in N_{*}(R)[x] = N_{*}(R[x])$ for each $0 \leq i \leq m$, $0 \leq j \leq n$. Thus, $R[x]$ is an $\alpha$-skew almost Armendariz ring. \\
Since $R$ is a subring of $R[x]$, therfore Converse is also true.
\end{proof}



\end{document}